\documentclass[12pt, 14paper,reqno]{amsart}
\usepackage{amsmath,amsfonts,amssymb}
\usepackage{longtable}

\usepackage[mathscr]{eucal}

\usepackage{amssymb, amsmath, amsthm}
\usepackage[breaklinks]{hyperref}

\usepackage{graphicx}
\usepackage[mathscr]{eucal}
\usepackage{mathrsfs}
\usepackage{amsbsy}
\usepackage{dsfont}
\usepackage{bbm}
\usepackage{wasysym}
\usepackage{stmaryrd}
\usepackage{url}

\input xypic
\xyoption{all}

\makeindex
\makeglossary

\begin{document}
\baselineskip = 16pt

\newcommand \ZZ {{\mathbb Z}}
\newcommand \NN {{\mathbb N}}
\newcommand \RR {{\mathbb R}}
\newcommand \PR {{\mathbb P}}
\newcommand \AF {{\mathbb A}}
\newcommand \GG {{\mathbb G}}
\newcommand \QQ {{\mathbb Q}}
\newcommand \CC {{\mathbb C}}
\newcommand \bcA {{\mathscr A}}
\newcommand \bcC {{\mathscr C}}
\newcommand \bcD {{\mathscr D}}
\newcommand \bcF {{\mathscr F}}
\newcommand \bcG {{\mathscr G}}
\newcommand \bcH {{\mathscr H}}
\newcommand \bcM {{\mathscr M}}
\newcommand \bcJ {{\mathscr J}}
\newcommand \bcL {{\mathscr L}}
\newcommand \bcO {{\mathscr O}}
\newcommand \bcP {{\mathscr P}}
\newcommand \bcQ {{\mathscr Q}}
\newcommand \bcR {{\mathscr R}}
\newcommand \bcS {{\mathscr S}}
\newcommand \bcV {{\mathscr V}}
\newcommand \bcW {{\mathscr W}}
\newcommand \bcX {{\mathscr X}}
\newcommand \bcY {{\mathscr Y}}
\newcommand \bcZ {{\mathscr Z}}
\newcommand \goa {{\mathfrak a}}
\newcommand \gob {{\mathfrak b}}
\newcommand \goc {{\mathfrak c}}
\newcommand \gom {{\mathfrak m}}
\newcommand \gon {{\mathfrak n}}
\newcommand \gop {{\mathfrak p}}
\newcommand \goq {{\mathfrak q}}
\newcommand \goQ {{\mathfrak Q}}
\newcommand \goP {{\mathfrak P}}
\newcommand \goM {{\mathfrak M}}
\newcommand \goN {{\mathfrak N}}
\newcommand \uno {{\mathbbm 1}}
\newcommand \Le {{\mathbbm L}}
\newcommand \Spec {{\rm {Spec}}}
\newcommand \Gr {{\rm {Gr}}}
\newcommand \Pic {{\rm {Pic}}}
\newcommand \Jac {{{J}}}
\newcommand \Alb {{\rm {Alb}}}
\newcommand \Corr {{Corr}}
\newcommand \Chow {{\mathscr C}}
\newcommand \Sym {{\rm {Sym}}}
\newcommand \Prym {{\rm {Prym}}}
\newcommand \cha {{\rm {char}}}
\newcommand \eff {{\rm {eff}}}
\newcommand \tr {{\rm {tr}}}
\newcommand \Tr {{\rm {Tr}}}
\newcommand \pr {{\rm {pr}}}
\newcommand \ev {{\it {ev}}}
\newcommand \cl {{\rm {cl}}}
\newcommand \interior {{\rm {Int}}}
\newcommand \sep {{\rm {sep}}}
\newcommand \td {{\rm {tdeg}}}
\newcommand \alg {{\rm {alg}}}
\newcommand \im {{\rm im}}
\newcommand \gr {{\rm {gr}}}
\newcommand \op {{\rm op}}
\newcommand \Hom {{\rm Hom}}
\newcommand \Hilb {{\rm Hilb}}
\newcommand \Sch {{\mathscr S\! }{\it ch}}
\newcommand \cHilb {{\mathscr H\! }{\it ilb}}
\newcommand \cHom {{\mathscr H\! }{\it om}}
\newcommand \colim {{{\rm colim}\, }} 
\newcommand \End {{\rm {End}}}
\newcommand \coker {{\rm {coker}}}
\newcommand \id {{\rm {id}}}
\newcommand \van {{\rm {van}}}
\newcommand \spc {{\rm {sp}}}
\newcommand \Ob {{\rm Ob}}
\newcommand \Aut {{\rm Aut}}
\newcommand \cor {{\rm {cor}}}
\newcommand \Cor {{\it {Corr}}}
\newcommand \res {{\rm {res}}}
\newcommand \red {{\rm{red}}}
\newcommand \Gal {{\rm {Gal}}}
\newcommand \PGL {{\rm {PGL}}}
\newcommand \Bl {{\rm {Bl}}}
\newcommand \Sing {{\rm {Sing}}}
\newcommand \spn {{\rm {span}}}
\newcommand \Nm {{\rm {Nm}}}
\newcommand \inv {{\rm {inv}}}
\newcommand \codim {{\rm {codim}}}
\newcommand \Div{{\rm{Div}}}
\newcommand \CH{{\rm{CH}}}
\newcommand \sg {{\Sigma }}
\newcommand \DM {{\sf DM}}
\newcommand \Gm {{{\mathbb G}_{\rm m}}}
\newcommand \tame {\rm {tame }}
\newcommand \znak {{\natural }}
\newcommand \lra {\longrightarrow}
\newcommand \hra {\hookrightarrow}
\newcommand \rra {\rightrightarrows}
\newcommand \ord {{\rm {ord}}}
\newcommand \Rat {{\mathscr Rat}}
\newcommand \rd {{\rm {red}}}
\newcommand \bSpec {{\bf {Spec}}}
\newcommand \Proj {{\rm {Proj}}}
\newcommand \pdiv {{\rm {div}}}
\newcommand \wt {\widetilde }
\newcommand \ac {\acute }
\newcommand \ch {\check }
\newcommand \ol {\overline }
\newcommand \Th {\Theta}
\newcommand \cAb {{\mathscr A\! }{\it b}}

\newenvironment{pf}{\par\noindent{\em Proof}.}{\hfill\framebox(6,6)
\par\medskip}

\newtheorem{theorem}[subsection]{Theorem}
\newtheorem{conjecture}[subsection]{Conjecture}
\newtheorem{proposition}[subsection]{Proposition}
\newtheorem{lemma}[subsection]{Lemma}
\newtheorem{remark}[subsection]{Remark}
\newtheorem{remarks}[subsection]{Remarks}
\newtheorem{definition}[subsection]{Definition}
\newtheorem{corollary}[subsection]{Corollary}
\newtheorem{example}[subsection]{Example}
\newtheorem{examples}[subsection]{examples}

\title{Elliptic surfaces to class groups and Selmer groups}
\author{Kalyan Banerjee, Kalyan Chakraborty and Azizul Hoque}
\address{K. Banerjee @Department of Mathematics, SRM University AP, Mangalagiri-Mandal, Amaravati-522502, Andhra Pradesh, India.} \email{kalyan.ba@srmap.edu.in}
\address{K. Chakraborty @Department of Mathematics, SRM University AP, Mangalagiri-Mandal, Amaravati-522502, Andhra Pradesh, India.} \email{kalyan.c@srmap.edu.in}
\address{A. Hoque @Department of Mathematics, Rangapara College (Autonomous), Rangapara, Sonitpur-784505, Assam, India.}
\email{ahoque.ms@gmail.com}
\date{\today}

\maketitle
\begin{abstract}
In this note, we connect the $n$-torsions of the Picard group of an elliptic surface to the $n$-divisibility of the class group of torsion fields for a given integer $n>1$. We also connect the $n$-divisibility of the Selmer group to that of the class group of torsion fields.
\end{abstract}

\section{Introduction}
One of the classical problems in algebraic number theory is to understand the structure of the class group of a given number field. It is known that the class group is finite for all number fields. The question is how to find an element of a given order in the class group of a number field. One of the approaches is to use algebro-geometric methods to find such elements. This idea of using algebraic geometry to find elements of large order was first introduced in the paper by Agboola and Pappus \cite{AP}. Later,  Gillibert and Levin \cite{GL} approached the problem by pulling back torsion line bundles to the class group. Recently, Gillibert \cite{GI} has shown how to pull back torsion elements on a hyperelliptic curve to the class group of quadratic number fields.

Let us briefly describe our main idea use in this article. Our approach to this problem is to start with a surface $S$ fibered over $\PR^1$ that admits elliptic fibration over $\bar\QQ$. That means that the general fibers of this fibration are elliptic curves. We consider the torison elements of the elliptic fibers of this fibration. If we vary the smooth fibers, we will have a curve fibered over $U\subset \PR^1$ (a Zariski open subset of $\PR^1$) such that over each closed point of $U$, we have the torsion of the corresponding fiber. We call this curve $C$ and spread the normalization of $C$ over $\Spec(\ZZ)$. Let it be denoted by $\bcC$. In addition, we consider a good prime $p$ such that the fiber $\bcC_p$ is a smooth arithmetic scheme. We consider the class group of this scheme.

To accomplish this, we start with the theory of Chow schemes and Hilbert schemes for arithmetic varieties, which parameterize cycles on an arithmetic variety, and then use 'etale monodromy of a smooth fibration to conclude that the torsion element mentioned above vary in a family. In this regard we would like to mention the result by \cite{GL} which also starts from elliptic surfaces and produces high-order elements in the class group of torsion fields, but our approach is different and uses Weil divisors and the Chow theory mentioned above.

The precise result is the following: 

\begin{theorem}\label{thm1}
Let $S\to \PR^1$ be an elliptic surface and ${\mathcal E}_S\to \PR^1$ be the corresponding Neron model. Then there exists an infinite family of $p$-torsion number fields corresponding to the $p$-torsions of the general smooth fibers such that the $p$-rank of the corresponding class groups remains constant. 
\end{theorem}

\section{Preliminary results}
We start with an example of  an elliptic surface $\mathcal{E}$ fibered over $\PR^1$ defined by $y^2=x^3+t$ with $t\in \mathbb{Q}$.  This gives a family of elliptic curves as $t$ varies over rational numbers. The elliptic curve $y^2=x^3+1$  has a $l$-torsion subgroup isomorphic to $\ZZ_l\times \ZZ_l$ here $l$ is a prime greater than or equal to $2$. Substituting $x=p$ we get $y^2=p^3+1$,  which corresponds to the quadratic field $\QQ(\sqrt{p^3+1})$. If we start with a $l$-torsion of the elliptic curve $y^2=x^3+1$, spread it out and specialize (as discussed earlier) which would entail an element in the class group of $\QQ(\sqrt{p^3+1})$. The important part of this construction is to prove the non-triviality of the elements obtained by this process.

We begin by considering a nontrivial element in the Chow group of codimension one algebraically trivial cycles modulo rational equivalence on $\mathcal {E}$ of degree zero (denoted by $A^1(\mathcal {E})$), which is canonically isomorphic to the Picard variety of $\mathcal {E}$. Let $\alpha$ be the element that is non-trivial $n$-torsion on $A^1(\mathcal {E})$. We then consider a fixed spread of the cycle $[\alpha]$ over $\Spec(\ZZ)$ and denote it by $\tilde{\alpha}$. Further we consider the fiber of the spread at some general scheme theoretic point $(P,Q)\in \Spec(\ZZ)\times_{\Spec \ZZ} \Spec(\ZZ)$, that is, denoted by $\tilde{\alpha}_{P,Q}$. This is a torsion element in the Chow group of the smooth arithmetic variety $\mathcal {E}_{P,Q}$, which contains $\bcO_K$ as a Zariski open set. Here $K=\QQ(\sqrt{p^3+n})$ and $\bcO_K$ is its ring of integers.
Now, restricting $\tilde{\alpha}_{P,Q}$ to $\bcO_K$ gives rise to an element in the class group of $K$.
Let $U\subset \Spec(\ZZ)\times \Spec(\ZZ)$ be the set of all pairs of primes in which the fibers are smooth arithmetic schemes defined over $\Spec(\ZZ)$. At this point, we recall the following result, which will be used in our construction.

\begin{theorem}[{\cite[Theorem 4.2]{BH}}] \label{BH24} The set
$$\bcZ_d:=\{(z,(P,Q))\in C^1_{d,d}(\mathcal {E_U}/U)|Supp(z)\subset {\mathcal{E}_{P,Q}}, n[z]=0\in \CH^1({\mathcal{E}_{P,Q}})\}$$ is a countable union of Zariski closed subsets in the Chow variety $C^1_{d,d}(\mathcal {E}_U/U)$ parametrizing the pairs of degree $d$ subvarieties of the arithmetic variety $\mathcal{E}_U$.
\end{theorem}

There are some crucial points to be noted here:

(I) The notion of Hilbert scheme and the Hom scheme makes sense for an arithmetic variety. This is as explained in \cite[Chapter: Hilbert schemes and Quot schemes, \S 5]{FGA}.

(II) The family of Weil divisors of a smooth fibration over $\Spec(\ZZ)$ is parameterized by a Chow variety, which is actually given by the Picard scheme parameterizing relative Cartier divisors of the same family \cite[Corollary 11.8]{Ry}. In our case the family is $\bcC_U$ which is of finite presentation over $\ZZ$ and is a standard smooth algebra\footnote{in the sense, stack exchange \cite{St}, Definitions 10.136.6 and 29.32.1.} over $\ZZ$. This enables us to formulate the definition of rational equivalence for arithmetic varieties as in \cite[\S 3.3]{GS} in the following way:

Two Weil divisors $D_1,D_2$ are rationally equivalent on a fiber $\bcC_b$, if there exists a morphism $$f: \PR^1_{U}\to C^1_{d,d}(\bcC_U/\PR^1_{U})$$ such that
$(f\circ 0)|_b=D_1+B$ and $ (f\circ \infty)|_b=D_2+B$,
where $B$ is a positive Weil divisor and $0,\infty$ are two fixed sections from $U$ to  $\PR^1_{U}$.

Let us assume that the divisor $D_b=D_b^+-D_b^-$ is rationally equivalent to zero. This means that there exists a map $$f:\PR^1\to C^1_{d,d}({\bcC_b})$$ such that
$$f(0)=D_b^{+}+\gamma\text{ and }f(\infty)=D_b^{-}+\gamma,$$
where $\gamma$ is a positive divisor on ${\bcC_b}$.
In other words, we have the following map:
$$\ev:Hom^v(\PR^1_{U},C^1_{d}(\bcC_U/U))\to C^1_{d}(\bcC_U/U)\times C^1_{d}(\bcC_U/U) $$
 given by $f\mapsto (f(0),f(\infty))$ and that the fiber of  $f$ at $b$ is contained in $C^1_{d,d}(\bcC_b)$.

We denote $C^1_{d}(\bcC_U/U)$ by $C^1_d(\bcC)$ for simplicity.
We now consider the subscheme $U_{v,d}(\bcC)$ of $\PR^1_{U}\times \Hom^v(\PR^1_{U},C^1_{d}(\bcC))$ consisting of pairs $(b,f)$ such that image of $f$ is contained in $C^1_{d}(\bcC_b)$ (such a universal family exists, for example, see \cite[Theorem 1.4]{Ko} or \cite[Chapter on Hilbert schemes and Quot schemes]{FGA}. This gives a morphism from $U_{v,d}(\bcC)$ to
$$\PR^1_{U}\times C^1_{d,d}(\bcC_b)$$
 defined by $$(b,f)\mapsto (b,f(0),f(\infty)).$$
Again, we consider the closed subscheme $\bcV_{d,d}$ of $\PR^1_{U}\times C^1_{d,d}(\bcC)$ given by $(b,z_1,z_2)$, where $(z_1,z_2)\in C^1_{d,d}(\bcC_b)$. Suppose that the map from $\bcV_{d,u,d,u}$ to $\bcV_{d+u,u,d+u,u}$ is given by
$$(A,C,B,D)\mapsto (A+C,C,B+D,D).$$
Let us denote the fiber product by $\bcV$ of $U_{v,d}(\bcC)$ and $\bcV_{d,u,d,u}$ over $\bcV_{d+u,u,d+u,u}$. If we consider the projection from $\bcV$ to $\PR^1_{U}\times C^1_{d,d}(\bcC)$, then we observe that $A$ and $B$ are supported and rationally equivalent in $\bcC_b$. Conversely, if $A$ and $B$ are supported as well as rationally equivalent on $\bcC_b$, then we get the map $$f:\PR^1_{U}\to C^1_{d+u,u,d+u,u}(\bcC)$$ of some degree $v$ satisfying
$$f(0)=(A+C,C)\text{ and } f(\infty)=(B+D,D),$$
where $C$ and $D$ are supported on $\bcC_b$. This implies that the image of the projection from $\bcV$ to $\PR^1_{U}\times C^1_{d,d}(\bcC)$ is a quasi-projective subscheme $W_{d}^{u,v}$ consisting of the tuples $(b,A,B)$ such that $A$ and $B$ are supported on $\bcC_b$, and that there exists a map $$f:\PR^1_{U}\to C^1_{d+u,u}(\bcC_b)$$ such that
$$f(0)=(A+C,C)$$
 and
 $$f(\infty)=(B+D,D)\;.$$
 Here $f$ is of degree $v$, and $C,D$ are supported on $\bcC_b$ and they are of co-dimension $1$ and degree $u$ cycles. This shows that $W_d$ is the union $\cup_{u,v} W_d^{u,v}$. We now prove that the Zariski closure of $W_d^{u,v}$ is in $W_d$ for each $u$ and $v$. For this, we prove the following:
$$W_d^{u,v}=pr_{1,2}(\wt{s}^{-1}(W^{0,v}_{d+u}\times W^{0,v}_u)),$$
where
$$\wt{s}: \PR^1_{U}\times C^1_{d,d,u,u}(\bcC)\to \PR^1_{U}\times C^1_{d+u,d+u,u,u}(\bcC)$$
defined by
$$\wt{s}(b,A,B,C,D)=(b,A+C,B+D,C,D).$$

We assume $(b,A,B,C,D)\in \PR^1_{U}\times C^1_{d,d,u,u}(\bcC)$ in such a way that $\wt{s}(b,A,B,C,D)\in W^{0,v}_{d+u}\times W^{0,v}_u$. This implies that there exists an element
$$(b,g)\in \PR^1_{U}\times\Hom^v(\PR^1_{U},C^p_{d+u}(\bcC))$$ and an element
$$(b,h)\in \Hom^v(\PR^1_{U},C^p_{u}(\bcC))$$ satisfying $$g(0)=A+C,~g(\infty)=B+D \text{ and } h(0)=C,h(\infty)=D$$ as well as the image of $g$ and $h$ are contained in $C^1_{d+u}(\bcC_b)$ and  $C^1_u(\bcC_b)$ respectively.

Also, if $f=g\times h$ then $f\in \Hom^v(\PR^1_{U},C^p_{d+u,u}(\bcC))$ is such that the image of $f$ is contained in $C^1_{d+u,u}(\bcC_b)$ and also satisfies the following:
$$f(0)=(A+C,C)\text{ and }(f(\infty))=(B+D,D).$$
This shows that $(b,A,B)\in W^d_{u,v}$.

On the other hand, if we assume that $(b,A,B)\in W^d_{u,v}$, then there exists $f\in \Hom^v(\PR^1_U,C^1_{d+u,u}(\bcC))$ such that
$$f(0)=(A+C,C)\text{ and }f(\infty)=(B+D,D),$$
and image of $f$ is contained in the Chow scheme of $\bar{\bcC_b}$.

We now compose $f$ with the projections to $C^1_{d+u}(\bcC_b)$ and to $C^1_{u}(\bcC_b)$ to get a map $g\in \Hom^v(\PR^1_{U},C^1_{d+u}(\bcC))$ and a map $h\in\Hom^v(\PR^1_{U},C^1_{u}(\bcC))$ satisfying
$$g(0)=A+C,\quad g(\infty)=B+D$$
and
$$h(0)=C,\quad h(\infty)=D.$$
Also, the image of $g$ and $h$ are contained in the respective Chow varities of the fibers $\bcC_b$. Therefore, we have
$$W_d=pr_{1,2}(\wt{s}^{-1}(W_{d+u}\times W_u)).$$

We are now in a position to prove that the closure of $W_d^{0,v} $ is contained in $W_d$. Let $(b,A,B)$ be a closed point in the closure of ${W_d^{0,v}}$. Let $W$ be an irreducible component of ${W_d^{0,v}}$ whose closure contains $(b,A,B)$. We assume that $U'$ is an affine neighborhood of $(b,A,B)$ such that $U'\cap W$ is nonempty. Then there is an irreducible curve $C'$ in $U'$ passing through $(b,A,B)$. Let $\bar{C'}$ be the Zariski closure of $C'$ in $\overline{W}$. The map
$$e:U_{v,d}(\bcC)\subset \PR^1_{U}\times \Hom^v(\PR^1_{U},C^1_{d}(\bcC))\to C^1_{d,d}(\bcC)$$
given by
$$(b,f)\mapsto (b,f(0),f(\infty))$$
is regular and $W_d^{0,v}$ is its image. We now choose a curve $T$ in $U_{v,d}(\bcC)$ such that the closure of $e(T)$ is $\bar C'$.  Let $\wt{T}$  denote the normalization of the Zariski closure of $T$, and $\wt{T_0}$ be the preimage of $T$ in this normalization. Then the regular morphism $\wt{T_0}\to T\to \bar C'$ extends to a regular morphism, when the scalar extends to the field of algebraic numbers. Let this morphism be $\wt{T}_{\QQ}$ to $\bar C'_{\QQ}$. If $(b_{\QQ},f_{\QQ})$ is a preimage of $(b_{\QQ},A_{\QQ},B_{\QQ})$, then
$$f_{\QQ}(0)=A_{\QQ}, \quad f_{\QQ}(\infty)=B_{\QQ}$$ and the image of $f_{\QQ}$ is contained in $C^p_{d}(C)$. 
Spreading out $f_{\QQ}$, we have an $f$ such that $$f(0)=A, \quad f(\infty)=B\;.$$
This is because there is a one-to-one correspondence between $\Spec(\bar\ZZ)$ points of arithmetic varieties and $\bar Q$ points of the corresponding variety over $\bar Q$. Therefore, $A$ and $B$ are rationally equivalent. This completes the proof.

As an application of the above result, we have the following.

\begin{theorem}\label{thm2}
The $l$-rank of the class group of the number fields $K=\QQ(\sqrt{p^3+n})$ obtained above remains constant in a family.
\end{theorem}

\begin{proof} It follows from Theorem \ref{BH24} that $\bcZ_d$ is a countable union of Zariski closed subsets in a parameter scheme. Let  $$\bcZ_d=\cup_{i=0}^{\infty}\bcZ_d^i\;,$$
and consider the family $\bcZ_d^i\to U$ for each $i$. Over an open subset $V$ of $U$, the above map is an \'etale morphism and let it be dominant too by \cite{BH}. Then we have an \'etale morphism which is surjective from $\bcZ_{d,V}\to V$. Now considering the corresponding fibrations with base $\bar \QQ$, we have a finite \'etale morphism from $\bcZ_{d,V,\QQ}\to V_{\QQ}$, if we replace $\bcZ_{d,V,\QQ}$ by a smooth multi-section (it is a multi-hyperplane section) over $\bar\QQ$. Then the fiber over each scheme-theoretic point forms a subgroup of elements of order $n$ in the class group of $\mathcal{E}_{P,Q}$. Since the morphism is \'etale and finite, it gives a local system over $V$ which is a sub-local system in the Tate module $T_l(E_{P,Q,\QQ})$ and hence the rank of the subgroup in the class groups mentioned above varies in a family in this sense. That is, there exists a subgroup in each of the class groups such that it's $l$-rank remains constant. If we can prove that the $l$-rank is non-zero in a certain class group among the members of the family above, then it is non-zero over a Zariski open subset of $\Spec(\ZZ)\times \Spec(\ZZ)$. 

Now consider the elliptic curve, $$y^2=x^3+n$$ for a fixed positive integer $n$. Assume that $m$ is a positive integer such that $-m^3+n<0$, then  by \cite[Theorem 4.1]{So}, the imaginary quadratic field $\QQ(\sqrt{n-m^3})$ has an element of order $p$ by specializing the $p$-torsion  in the elliptic curve (defined over $\QQ$) to the class group of the imaginary quadratic field given above. Now let us calculate the torsion group defined over $\QQ$ of the elliptic curve $$y^2=x^3+1.$$ By Nagell-Lutz theorem, either $$y=0$$ or $$y^2|\Delta$$ where $\Delta=27$ the discriminant of the given elliptic curve. Then $$y^2=1,9$$ are the only possibilities or $y=0$. In the case $y=0$ we get that $x^3+1=0$ and hence $$x=-1, \frac{1\pm\sqrt{3}}{2}$$ is a root of this equation. Therefore there is a $2$-torsion $(-1,0)$ defined over rational numbers. 

In the other cases, we have $$x^3+1=1,$$
that is, $x=0$ and $y=\pm 1$, or 
$$x^3+1=9$$ and hence $x^3=8$ implying $x=2$, so $$(2,3), (2,-3).$$

Thus the torsion subgroup is 
$$\{(-1,0), (2,3), (2,-3), (0,1), (0,-1), (0,\frac{1\pm\sqrt{3}}{2}),\bcO\},$$ so it is a subgroup of order $8$ and it contains the cyclic group $\ZZ_2$ and it is 
$$\ZZ_2\times\ZZ_4.$$
Therefore, there is a $2$ torsion and a 4-torsion in the class group of the number fields $\QQ(\sqrt{1-m^3})$ for $m>1$. 

Now varying $n$, that is replacing $1$ by $n$ in the above equation we get the elliptic curve $$y^2=x^3+n.$$ Since the torsion subgroup varies in a family of imaginary quadratic fields, we have  a subgroup of order $8$ in the class group of $\QQ(\sqrt{n-m^3})$ by the technique of \cite{So}. Moreover by our technique the subgroup varies in a family of  imaginary quadratic fields $\QQ(\sqrt{n-m^3})$.

On the other hand, if we fix the constant $n=1$ and vary $m$ in  the above equation we get the family $\QQ(\sqrt{1-m^3})$ whose class group has a subgroup $\ZZ_2\times \ZZ_4$.
\end{proof}

\section{Proof of the main theorem \ref{thm1}}
Let $C$ be the curve closely embedded in $\mathcal{E}$ such that it is fibered over $\PR^1$ and each fiber consists of $p$-torsions of the smooth fiber $\mathcal{E}_b$ for a closed point $b\in \PR^1$. When we consider the fibers it gives us the family of torsion fields that attach the torsion points of order $p$ with $\QQ$. In this section, we are interested in the class group of this family of torsion fields. 

Consider a torsion point on $\Pic^0(C)$. Then we obtain a fixed spread $\bcC$ of the curve $C$, the Neron model $\Pic^0(\bcC)$ of the Picard variety, and the spread of $[\alpha]$, say $\widetilde{\alpha}$. Now consider a good prime $P$ in $\Spec{\ZZ}$ such that the specialization $\bcC_P$ is a smooth arithmetic scheme and it is the ring of integers of the torsion field corresponding to $\mathcal{E}_P$. Specializing $\widetilde{\alpha}$ in the fiber $\bcC_P$ we will have $\widetilde{\alpha}_P$ in the class group of the torsion field over the point $P$. Our aim is to prove that this element in the class group is nonzero for a suitable element $\alpha$. By an analog of Theorem \ref{thm2}, the $p$-rank remains constant in a family of torsion fields. Theorem \ref{thm1} can be rewritten in this configuration as follows:

\begin{theorem}\label{thm3}
Assume an element of order $n$ in the $p$-Selmer group of the Jacobian of $C$. Then it gives an $n$-torsion  in the class group of the torsion fields. 
\end{theorem}
 
\begin{proof} Let us consider the push-forward map from $$\Pic^0(\mathcal E)\to \Pic^0(C)$$ Note that both are Galois modules with the action of the Galois group of $\bar\QQ/\QQ$. The Galois action on $\Pic^0(C)$ is compatible with the Galois action on the torsion fields by  Fulton's intersection theory about specialization homomorphisms (cf. \cite{Ful}). Thus there exists a Galois module structure on $cl(\bcC_P)$ (class group of $\bcC_P$ as discussed earlier). So we have the homomorphism by our construction (as mentioend in the introduction) from $$\Pic^0(C)\to cl(\bcC_P)$$ for a general point $P$. Further taking the Galois action into account we have a functorial homomorphism $$H^1(G, \Pic^0(C))\to H^1(G,cl(\bcC_P)).$$
 
 Similarly we have a functorial homomorphism from the $n$-Selmer group of $\Pic^0(C)$ to the $n$-selmer group of the class group of $\bcC_P$. Since the later group is just $H^1(G,cl(\bcC_P))$, we have a homomorphism from 
$$\Pic^0(C)(\QQ)/n\Pic^0(C)(\QQ)\to cl(\bcC_P)^G/n cl(\bcC_P)^G.$$
Taking $n$ large we have a map from 
$$\Pic^0(C)(\QQ)/n\Pic^0(C)(\QQ)\to cl(\bcC_P)^G\subset \cl(\bcC_P).$$
Further composing with the map $$E_P\to Alb(\mathcal E)\cong\Pic^0(\mathcal E)\to \Pic^0(C)$$
we actually have a homomorphism from the $n$-Selmer group of $E_P$ to the $\cl(\bcC_P)$.
\end{proof}

\subsection{A toy computation}
Let us consider the previous example $y^2=x^3+t$ and consider the curve $C$ closely embedded in $\mathcal{E}$ which parametrizes the $2$ torsions in the fiber. For example consider the smooth fiber $$ y^2 = x^3 + 17$$ and the $2$-torsion field for this curve is isomorphic to $\QQ({17}^{1/3})$. Now, by the Nagell-Lutz Theorem, we have $$ y^2|\Delta$$
where $\Delta$ is the discriminant of the above elliptic curve. Here $$\Delta=27.17^3,$$ so we have $$y^2=1,9, 17^2, 17^2.9.$$ For the above solutions the values of $x$ are given by $$x^3=-8$$ so $$x=-2, 1\pm \sqrt{3} $$ and thus  the torsion points are $$\{\bcO, (-2,\pm 3), (1\pm \sqrt{3}, \pm 3)\},$$ which are defined over the integers and it is isomorphic to $$\ZZ_3\times \ZZ_3$$ or $$\ZZ_9.$$

Note that this subgroup of integer torsion points is embedded in $$E(\QQ)/2E(\QQ)$$ which is also embedded in the 2-Selmer group of $E$, $E$ is given by $$y^2=x^3+17\;.$$

Let us denote this subgroup by $H$ and by the previous technique the subgroup $H$ is specialized to the class group of the number field $\QQ({-17}^{1/3})$. So, first of all, the Selmer group of the elliptic curve $E: y^2=x^3+17$ has a subgroup of order $9$. This subgroup, when specialized to the class group of cubic fields, may give elements of order $3$ or of order $9$ in the corresponding cubic field. 

\subsection*{Acknowledgements} The first two authors would like to thank the AP of SRM University for support and providing a congenial atmosphere to carry out this research. The third author is supported by the SERB MATRICS grant (No. MTR/2021/000762) and the ANRF(SERB) CRG grant (No. CRG/2023/007323), Govt. of India. The authors are grateful to the anonymous referees for careful reading of the paper and for valuable comments/suggestions which have helped improving the presentation immensely.



\end{document}